\documentclass[12pt,reqno]{amsart}

\usepackage{amssymb}

\numberwithin{equation}{section}

\newtheorem{theorem}{Theorem}[section]
\newtheorem{proposition}[theorem]{Proposition}
\newtheorem{lemma}[theorem]{Lemma}

\theoremstyle{definition}

\newtheorem{remark}[theorem]{Remark}

\begin{document}

\baselineskip=15pt

\title[Vector bundles on curves defined over number fields]{A Belyi-type criterion
for vector bundles on curves defined over a number field}

\author[I. Biswas]{Indranil Biswas}

\address{Department of Mathematics, Shiv Nadar University, NH91, Tehsil Dadri,
Greater Noida, Uttar Pradesh 201314, India}

\email{indranil.biswas@snu.edu.in, indranil29@gmail.com}

\author[S. Gurjar]{Sudarshan Gurjar}

\address{Department of Mathematics, Indian Institute of Technology Bombay,
Powai, Mumbai 400076, Maharashtra, India}

\email{sgurjar@math.iitb.ac.in}



\date{}

\begin{abstract}
Let $X_0$ be an irreducible smooth projective curve defined over $\overline{\mathbb Q}$ and
$f_0\, :\, X_0\, \, \longrightarrow\, \mathbb{P}^1_{\overline{\mathbb Q}}$ 
a nonconstant morphism whose branch locus is contained in the subset
$\{0,\,1,\, \infty\} \, \subset\, \mathbb{P}^1_{\overline{\mathbb Q}}$. For any vector
bundle $E$ on $X\,=\, X_0\times_{{\rm Spec}\,\overline{\mathbb Q}} {\rm Spec}\,\mathbb{C}$, consider the
direct image $f_*E$ on $\mathbb{P}^1_{\mathbb C}$, where $f\,=\, (f_0)_{\mathbb C}$. It decomposes into
a direct sum of line bundles and also it has a natural parabolic structure. We prove that $E$ is the
base change, to $\mathbb C$, of a vector bundle on $X_0$ if and only if there is an isomorphism
$f_*E \, \stackrel{\sim}{\longrightarrow}\, \bigoplus_{i=1}^r {\mathcal O}_{{\mathbb P}^1_{\mathbb C}}(m_i)$,
where $r\,=\, {\rm rank}(f_*E)$, that takes the parabolic structure on $f_*E$ to a parabolic structure on
$\bigoplus_{i=1}^r {\mathcal O}_{{\mathbb P}^1_{\mathbb C}}(m_i)$ defined over $\overline{\mathbb Q}$.
\end{abstract}

\maketitle

\section{Introduction}

A well-known theorem of Gennadii V. Belyi says that an irreducible smooth complex projective curve $X$ is isomorphic
to one defined over $\overline{\mathbb Q}$ if and only if $X$ admits a nonconstant morphism to $\mathbb{P}^1_{\mathbb C}$
whose branch locus is contained in the subset $\{0,\,1,\, \infty\} \, \subset\, \mathbb{P}^1_{\mathbb C}$. It can
be deduced from a work of Weil, \cite{We}, that if $X$ admits a nonconstant morphism $f$ to $\mathbb{P}^1_{\mathbb C}$
whose branch locus is contained in $\{0,\,1,\, \infty\}$, then $X$ is isomorphic to a curve defined over $\overline{\mathbb Q}$
such that $f$ is also defined over $\overline{\mathbb Q}$ (see also \cite{Go1}). But the converse, namely 
$X$ admits a nonconstant morphism to $\mathbb{P}^1_{\mathbb C}$,
whose branch locus is contained in $\{0,\,1,\, \infty\}$, if $X$ is isomorphic to a curve
defined over $\overline{\mathbb Q}$, involves a very ingenious construction
of Belyi. See \cite{Go2} for a result along this line for complex surfaces.
See \cite{Gr2}, \cite{SL} for a program inspired by the work of Belyi.

Let $X$ be an irreducible smooth complex projective curve which is isomorphic to one defined over $\overline{\mathbb Q}$.
Our aim here is to address the following question:

Given a vector bundle on $X$, when is it isomorphic to one defined over $\overline{\mathbb Q}$? To formulate this question more
precisely, let $X_0$ be an irreducible smooth projective curve defined over $\overline{\mathbb Q}$. Let
$E$ be a vector bundle on the complex projective curve $X\,=\, (X_0)_{\mathbb C}\, :=\,
X_0\times_{{\rm Spec}\,\overline{\mathbb Q}} {\rm Spec}\,\mathbb{C}$. The question is to decide whether $E$ is isomorphic
to the base change, to $\mathbb C$, of a vector bundle over $X_0$.

Using Belyi's criterion, fix a nonconstant morphism
$$
f_0\, :\, X_0\, \, \longrightarrow\, \mathbb{P}^1_{\overline{\mathbb Q}}
$$
whose Branch locus is contained in $\{0,\,1,\, \infty\} \, \subset\, \mathbb{P}^1_{\overline{\mathbb Q}}$. Let
$$
f\,=\, (f_0)_{\mathbb C} \,:\, X\,=\, (X_0)_{\mathbb C}\, \longrightarrow\, \mathbb{P}^1_{\mathbb C}
$$
be the base change of $f_0$. The direct image $f_*E\, \longrightarrow\, \mathbb{P}^1_{\mathbb C}$ splits into a direct
sum of line bundle \cite{Gr1}. Consequently, $f_*E$ is isomorphic to the base change, to $\mathbb C$, of a vector bundle on
$\mathbb{P}^1_{\overline{\mathbb Q}}$. Let $\bigoplus_{i=1}^r {\mathcal O}_{{\mathbb P}^1_{\overline{\mathbb Q}}}(m_i)$
be the vector bundle on $\mathbb{P}^1_{\overline{\mathbb Q}}$ admitting an isomorphism
\begin{equation}\label{e0}
\Psi\, :\, f_*E \, \longrightarrow\, \left(\bigoplus_{i=1}^r {\mathcal O}_{{\mathbb P}^1_{\overline{\mathbb Q}}}(m_i)\right)
\otimes_{\overline{\mathbb Q}}{\mathbb C}\,=\, \bigoplus_{i=1}^r {\mathcal O}_{{\mathbb P}^1_{\mathbb C}}(m_i),
\end{equation}
where $r\,=\, {\rm rank}(f_*E)$.

The direct image $f_*E$ has a natural parabolic structure; parabolic vector bundles were introduced in \cite{MS} (their
definition is recalled in Section \ref{se2.2}). Using the isomorphism $\Psi$ in \eqref{e0}, the parabolic structure on $f_*E$
produces a parabolic structure on $\bigoplus_{i=1}^r {\mathcal O}_{{\mathbb P}^1_{\mathbb C}}(m_i)$.
We prove the following (see Proposition \ref{prop1} and Theorem \ref{thm1}):

\begin{theorem}\label{thm0}
A vector bundle $E\, \longrightarrow\, X$ is isomorphic to the base change, to $\mathbb C$, of a vector bundle over
$X_0$ if and only if there is an isomorphism $\Psi$ as in \eqref{e0} such that the corresponding parabolic structure
on $\bigoplus_{i=1}^r {\mathcal O}_{{\mathbb P}^1_{\mathbb C}}(m_i)$ is defined over $\overline{\mathbb Q}$.
\end{theorem}

\section{Direct image and parabolic structure}

\subsection{Direct image on the projective line}\label{se2.1}

Let $X_0$ be an irreducible smooth projective curve defined over $\overline{\mathbb Q}$. Let
\begin{equation}\label{e1a}
f_0\, :\, X_0\, \, \longrightarrow\, \mathbb{P}^1_{\overline{\mathbb Q}}
\end{equation}
be a nonconstant morphism which is unramified over the complement
$\mathbb{P}^1_{\overline{\mathbb Q}}\setminus \{0,\,1,\, \infty\}$. In other
words, the branch locus of $f$ is contained in the subset $\{0,\,1,\, \infty\}\,
\subset\, \mathbb{P}^1_{\overline{\mathbb Q}}$.

Let
$$
X\,=\, (X_0)_{\mathbb C}\, =\, X_0\times_{{\rm Spec}\,\overline{\mathbb Q}} {\rm Spec}\,\mathbb{C}
$$
be the base change of $X_0$ to $\mathbb C$. Let
\begin{equation}\label{e2b}
f\,\,:=\,\, (f_0)_{\mathbb C}\,\, :\,\,
X\,\, \longrightarrow\, \,\mathbb{P}^1_{\overline{\mathbb Q}}\times_{{\rm Spec}\,\overline{\mathbb Q}}
{\rm Spec}\,{\mathbb C}\,\,=\,\, \mathbb{P}^1_{\mathbb C}
\end{equation}
be the base change, to $\mathbb C$, of the map $f_0$ in \eqref{e1a}. So $f$ is unramified
over the complement $\mathbb{P}^1_{\mathbb C}\setminus \{0,\,1,\, \infty\}$.

Let $E$ be a vector bundle over the smooth complex projective curve $X$. Consider
the direct image
\begin{equation}\label{e3}
W\, \,:=\, f_*E\,\, \longrightarrow\,\, \mathbb{P}^1_{\mathbb C}
\end{equation}
under the map $f$ in \eqref{e2b}. This vector bundle $W$ on $\mathbb{P}^1_{\mathbb C}$ has a natural
parabolic structure over $\{0,\,1,\, \infty\}$. Parabolic vector bundles were introduced in \cite{MS};
a natural parabolic structure on a direct image was constructed in \cite[Section~4]{AB}. We briefly recall the
definition of a parabolic bundle and also the construction of a parabolic structure on a direct image.

\subsection{Parabolic bundles and direct image}\label{se2.2}

Let $X$ be any compact connected Riemann surface. Let
$$
D\,\,:=\,\, \{x_1,\, \cdots,\, x_\ell\}\,\, \subset\,\, X
$$
be a finite subset.
Take a holomorphic vector bundle $E$ on $X$. A \textit{quasi-parabolic structure} on $E$ is a
strictly decreasing filtration of subspaces
\begin{equation}\label{e1}
E_{x_i}\,=\, E^1_i\, \supsetneq\, E^2_i \,\supsetneq\, \cdots\, \supsetneq\,
E^{n_i}_i \, \supsetneq\, E^{n_i+1}_i \,=\, 0
\end{equation}
for every $1\, \leq\, i\, \leq\, \ell$; here $E_{x_i}$ denotes the fiber
of $E$ over the point $x_i\,\in\, D$. A \textit{parabolic structure} on $E$ is a
quasi-parabolic structure as above together with $\ell$ increasing sequences of rational numbers
\begin{equation}\label{e2}
0\, \leq\, \alpha_{i,1}\, <\, \alpha_{i,2}\, <\,
\cdots\, < \, \alpha_{i,n_i}\, < 1\, , \ \ 1\, \leq\, i\, \leq\, \ell \, ;
\end{equation}
the rational number $\alpha_{i,j}$ is called the \textit{parabolic weight} of the subspace $E^j_i$ in
the quasi-parabolic filtration in \eqref{e1}. The \textit{multiplicity} of a parabolic weight $\alpha_{i,j}$ at
$x_i$ is defined to be the dimension of the complex vector space $E^j_i/E^{j+1}_i$.
A parabolic vector bundle is a holomorphic vector bundle equipped with a parabolic structure.
The subset $D$ is called the \textit{parabolic divisor}. (See \cite{MS}, \cite{MY}.)

Let $X$ and $Y$ be compact connected Riemann surfaces and
\begin{equation}\label{ep}
\phi\,\,:\,\, X \,\, \longrightarrow\,\, Y
\end{equation}
a nonconstant holomorphic map. Let
\begin{equation}\label{e8}
R\, \subset\, X
\end{equation}
be the ramification locus of $\phi$. For any point $x\,\in\, X$, let $m_x\, \geq\, 1$ be the multiplicity
of $\phi$ at $x$, so $m_x\, \geq\, 2$ if and only if $x\,\in\, R$. Let
\begin{equation}\label{e9}
\Delta\,=\, \phi (R) \, \subset\, Y.
\end{equation}
Let $E$ be a holomorphic vector bundle on $X$. We will construct a parabolic structure on the direct image
$\phi_*E$ whose parabolic divisor is the finite subset $\Delta$ defined in \eqref{e9}.

We recall a general property of a direct image. For any point $y\, \in\, Y$, the fiber $(\phi_* E)_y$
of $\phi_* E$ over $y$ has a certain canonical decomposition
\begin{equation}\label{e11}
(\phi_* E)_y\,=\, \bigoplus_{x\,\in \,\phi^{-1}(y)} V_x
\end{equation}
such that $\dim V_x\,=\, m_x\cdot \text{rank}(E)$, where $m_x$ is the multiplicity of $\phi$ at $x$ (see \cite[p.~19562, (4.4)]{AB}).
To describe the subspace $V_x\, \subset\, (\phi_* E)_y$ in \eqref{e11}, consider the homomorphism
\begin{equation}\label{h1}
\phi_*\left(E\otimes\left(\bigotimes_{z \in \phi^{-1}(y)\setminus x} {\mathcal O}_X(-m_z z)\right)\right)\,\, \longrightarrow\,\, \phi_*E
\end{equation}
given by the natural inclusion of $E\otimes \left(\bigotimes_{z \in \phi^{-1}(y)\setminus x} {\mathcal O}_X(-m_z z)\right)$
in $E$. The subspace $V_x\,\subset\, (\phi_* E)_y$ is the image of the homomorphism of fibers
$$
\left(\phi_*\left(E\otimes\left(\bigotimes_{z \in \phi^{-1}(y)\setminus x} {\mathcal O}_X(-m_z z)\right)\right)\right)_y
\, \longrightarrow\, (\phi_*E)_y
$$
corresponding to the homomorphism of coherent analytic sheaves in \eqref{h1}.

We will also recall an explicit description of the subspace $V_x\, \subset\, (\phi_* E)_y$. Take any
small open disk $x\, \in\, U\, \subset\, X$ around $x$ such that
\begin{itemize}
\item $U\bigcap \phi^{-1}(y)\,=\, \{x\}$,

\item $U\bigcap R\, \subset\, \{x\}$, and

\item $\# \phi^{-1}(y')\bigcap U\,=\, m_x$ for all $y'\, \in\, \phi(U)\setminus \{y\}$.
\end{itemize}
Let
\begin{equation}\label{e12}
\psi\, :=\, \phi\big\vert_U \, :\, U \, \longrightarrow\, \phi(U)
\end{equation}
be the restriction of $\phi$ to $U$. There is natural a homomorphism
\begin{equation}\label{uh}
\rho^x\, :\, \psi_* (E\big\vert_U)\, \longrightarrow\, (\phi_* E)\big\vert_{\psi(U)}
\end{equation}
arising from the commutative diagram of maps
$$
\begin{matrix}
U & \hookrightarrow & X\\
\,\,\, \Big\downarrow \psi && \,\,\, \Big\downarrow \phi\\
\phi(U) & \hookrightarrow & Y
\end{matrix}
$$
Let
\begin{equation}\label{uh2}
\rho^x_y\,\,:\,\, (\psi_* (E\big\vert_U))_y \,\, \longrightarrow\,\, (\phi_* E)_y
\end{equation}
be the homomorphism of fibers over $y$ corresponding to the homomorphism of coherent analytic sheaves in
\eqref{uh}. The restriction of $\phi_* E$ to a sufficiently small open neighborhood of $y\,\in\,
Y$ is the direct sum $\bigoplus_{x\in\phi^{-1}(y)} {\rm image}(\rho^x)$ (see \eqref{uh}). From this
it follows immediately that $\rho^x_y$ in \eqref{uh2} is fiberwise injective.
The subspace $\rho^x_y((\psi_* (E\big\vert_U))_y)\, \subset\, (\phi_* E)_y$ coincides with
$V_x$. Now we have the decomposition in \eqref{e11}.

The parabolic structure on $(\phi_* E)_y$ will be described by giving a parabolic structure on each
direct summand $V_x$ and then taking their direct sum. To give a parabolic structure on $V_x$, first
note that for any $j\, \geq\, 0$, there is a natural injective homomorphism of coherent analytic sheaves
\begin{equation}\label{h-1}
\phi_*\left(E\otimes {\mathcal O}_X(-jx) \otimes\left(\bigotimes_{z \in 
\phi^{-1}(y)\setminus x} {\mathcal O}_X(-m_z z) \right)\right)\,\, \longrightarrow\,\, \phi_*E
\end{equation} (see \eqref{h1}).
The image of the fiber $\phi_*\left(E\otimes {\mathcal O}_X(-jx) \otimes\left(\bigotimes_{z \in
\phi^{-1}(y)\setminus x} {\mathcal O}_X(-m_z z) 
\right)\right)_y$ in $(\phi_* E)_y$ by the homomorphism in \eqref{h-1} will be denoted by $\mathbf{E}(x,j)$.
We have a filtration of subspaces of $V_x$:
\begin{equation}\label{e13}
V_x\,:=\, \mathbf{E}(x,0) \, \supset \, \mathbf{E}(x,1) \, \supset \, \mathbf{E}(x,2) \, 
\supset\, \cdots\, \supset\,\mathbf{E}(x, m_x-1)\, \supset\, \mathbf{E}(x, m_x) \,=\,0.
\end{equation}
Note that 
$$
\phi_*\left(E\otimes\left(\bigotimes_{z \in \phi^{-1}(y)} {\mathcal O}_X(-m_z z) \right)\right)\,\,=\,\,
(\phi_*E)\otimes {\mathcal O}_Y(-y)
$$
by the projection formula, and hence we have $\mathbf{E}(x,m_x) \,=\,0$. The parabolic weight of the subspace 
\begin{equation}\label{a1}
\mathbf{E}(x,k)\, \subset\, V_x
\end{equation}
in \eqref{e13} is $\frac{k}{m_x}$.

This way we have a parabolic structure on the vector space $V_x$. Now taking the direct sum of these parabolic 
structures we get a parabolic structure on $E_y$ using \eqref{e11}.

\subsection{A property of the parabolic structure}\label{se2.3}

We return to the set-up of Section \ref{se2.1}. For any $m\, \in\, \mathbb Z$, the line bundle on ${\mathbb 
P}^1_{\mathbb C}$ of degree $m$ will be denoted by ${\mathcal O}_{{\mathbb P}^1_{\mathbb C}}(m)$. Any vector 
bundle over ${\mathbb P}^1_{\mathbb C}$ of rank $r$ decomposes into a direct sum of the form $\bigoplus_{i=1}^r
{\mathcal O}_{{\mathbb P}^1_{\mathbb C}}(m_i)$ \cite[p.~122, Th\'eor\`eme 1.1]{Gr1}. Therefore, every vector bundle
over ${\mathbb P}^1_{\mathbb C}$ is isomorphic to the base change, to $\mathbb C$, of a vector bundle defined over
$\mathbb{P}^1_{\overline{\mathbb Q}}$.

Let $E_0$ be a vector bundle over $X_0$. Consider the direct image
\begin{equation}\label{w0}
W_0\,\, :=\,\, (f_0)_*E_0\,\, \longrightarrow\,\,\mathbb{P}^1_{\overline{\mathbb Q}},
\end{equation}
where $f_0$ is the map in \eqref{e1a}. Set $E$ in \eqref{e3} to be the vector bundle
\begin{equation}\label{a2}
E\,:=\, E_0\otimes_{\overline{\mathbb Q}} \mathbb{C}\, \longrightarrow\, X\,=\,
X_0\times_{{\rm Spec}\,\overline{\mathbb Q}} {\rm Spec}\,\mathbb{C}
\end{equation}
obtained by base change of $E_0$ to $\mathbb C$. Therefore, the direct image
$W\,=\,f_*E$ (as in \eqref{e3}) of $E$ in \eqref{a2} is the base change
\begin{equation}\label{f1}
W\,\,=\,\, f_*E \,\,=\,\, ((f_0)_*E_0)\otimes_{\overline{\mathbb Q}} \mathbb{C}
\,\,=\,\, W_0\otimes_{\overline{\mathbb Q}} \mathbb{C}
\end{equation}
of $W_0$ (see \eqref{w0}) to $\mathbb C$.

\begin{proposition}\label{prop1}
The parabolic structure on the direct image $W$ in \eqref{f1} is defined over $\overline{\mathbb Q}$;
in other words, this parabolic structure is given by a parabolic structure on $W_0$ (defined in \eqref{w0}).
\end{proposition}

\begin{proof}
Note that all the ramification points in $X$ for the map $f$ are defined over $\overline{\mathbb Q}$.
Since $E$ is the base change to $\mathbb C$ of $E_0$, for any $x\, \in\, f^{-1}(\{0,\,1,\, \infty\})$
and any $j\, \geq\, 0$, the vector bundle 
$$
f_*\left(E\otimes {\mathcal O}_X(-jx) \otimes\left(\bigotimes_{z \in 
f^{-1}_0(f_0(x)) \setminus x} {\mathcal O}_X(-m_z z) \right)\right)
$$
(see \eqref{h-1}) satisfies the following condition:
$$
f_*\left(E\otimes {\mathcal O}_X(-jx) \otimes\left(\bigotimes_{z \in 
f^{-1}_0(f_0(x)) \setminus x} {\mathcal O}_X(-m_z z) \right)\right)\,\,=
$$
$$
(f_0)_*\left(E_0\otimes {\mathcal O}_{X_0}(-jx) \otimes\left(\bigotimes_{z \in 
f^{-1}_0(f_0(x)) \setminus x} {\mathcal O}_{X_0}(-m_z z) \right)\right)\otimes_{\overline{\mathbb Q}} \mathbb{C},
$$
where $E_0$ and $f_0$ are as in \eqref{w0}.
In view of this, the proposition is evident from the construction of the parabolic structure on the direct image
$W$ (see Section \ref{se2.2}).
\end{proof}

In the next section we will prove a converse of Proposition \ref{prop1}.

\section{Parabolic structure on a pullback}

\subsection{Parabolic structure defined over $\overline{\mathbb Q}$}

Let
\begin{equation}\label{e3b}
E\,\, \longrightarrow\,\,\, X\,\,=\,\, X_0\times_{{\rm Spec}\,\overline{\mathbb Q}} {\rm Spec}\,\mathbb{C}
\end{equation}
be a holomorphic vector bundle. Consider the parabolic structure on the direct image $W$ defined as in \eqref{e3}. This
parabolic bundle will be denoted by
\begin{equation}\label{e23}
W_*.
\end{equation}
As noted in Section \ref{se2.3}, $W$ is isomorphic to a vector
bundle defined over $\overline{\mathbb Q}$. Let ${\mathcal W}\, \longrightarrow\,
{\mathbb P}^1_{\overline{\mathbb Q}}$ be a vector bundle and
\begin{equation}\label{e14}
\Psi\,\,:\,\, W \, \longrightarrow\, {\mathcal W}\otimes_{\overline{\mathbb Q}}{\mathbb C}
\end{equation}
an isomorphism.

Using $\Psi$ in \eqref{e14}, we will consider $W$ to be the base change, to $\mathbb C$, of
the vector bundle $\mathcal W$ defined over $\overline{\mathbb Q}$.
Since the point $0,\, 1,\, \infty$ are defined over $\overline{\mathbb Q}$,
and $W\,=\, {\mathcal W}\otimes_{\overline{\mathbb Q}}{\mathbb C}$, it makes sense to ask
whether the quasiparabolic filtrations of $W_*$ are given by quasiparabolic filtrations
on $\mathcal W$, or in other words, whether the parabolic structure on $W$ is defined over
$\overline{\mathbb Q}$ (recall that the parabolic weights of $W_*$ are rational numbers).

The following is the main result proved here.

\begin{theorem}\label{thm1}
If the quasiparabolic filtrations of $W_*$ are defined over $\overline{\mathbb Q}$, then $E$ is
isomorphic to the base change, to $\mathbb C$, of a vector bundle on $X_0$ (see \eqref{e1a})
defined over $\overline{\mathbb Q}$.
\end{theorem}

\begin{proof}
Let
\begin{equation}\label{e15}
\varphi_0\,\, :\,\, Y_0\,\, \longrightarrow\, \,\mathbb{P}^1_{\overline{\mathbb Q}}
\end{equation}
be the Galois closure of the map $f_0$ in \eqref{e1a}. Let
\begin{equation}\label{e16}
\Gamma\,\, :=\,\, \text{Gal}(\varphi_0)\, \,=\,\, \text{Aut}(Y_0/\mathbb{P}^1_{\overline{\mathbb Q}})
\end{equation}
be the Galois group for the map $\varphi_0$ in \eqref{e15}. Let
\begin{equation}\label{e17}
\gamma_0\,\,:\,\, Y_0\,\, \longrightarrow\,\, X_0
\end{equation}
be the natural map, so we have
\begin{equation}\label{e17a}
f_0\circ \gamma_0\,=\, \varphi_0,
\end{equation}
where $f_0$ is the map in \eqref{e1a}. Let $Y\,=\, Y_0\times_{{\rm Spec}\,\overline{\mathbb Q}}{\rm Spec}\,{\mathbb C}$ be
the base change of $Y_0$ to $\mathbb C$. Let
\begin{equation}\label{e15b}
\varphi\, :\, Y\, \longrightarrow\,\mathbb{P}^1_{\mathbb C}\ \ \text{ and }\ \
\gamma\,:\, Y\, \longrightarrow\, X
\end{equation}
be the base changes, to $\mathbb C$, of $\varphi_0$ and $\gamma_0$ respectively.

Given any nonconstant holomorphic map $\delta\, :\, Z_1\, \longrightarrow\, Z_2$ between compact connected
Riemann surfaces, and a parabolic vector bundle $V_*$ on $Z_2$, we have the pulled back parabolic vector bundle
$\delta^*V_*$ on $Z_1$; see \cite[Section~3]{AB}. The parabolic divisor for $\delta^*V_*$ is the reduced inverse image
$\delta^{-1}(D_V)_{\rm red}$, where $D_V$ is the parabolic divisor for $V_*$. If $Z_1,\, Z_2$ and $\delta$ are defined over
$\overline{\mathbb Q}$, and the parabolic vector bundle $V_*$ is also defined over $\overline{\mathbb Q}$, then
the pulled back parabolic vector bundle $\delta^*V_*$ over $Z_1$ is also defined over $\overline{\mathbb Q}$.
Indeed, this follows immediately from the construction of the parabolic bundle $\delta^*V_*$.

Take any holomorphic vector bundle $V$ on $Y\,=\, Y_0\times_{{\rm Spec}\,\overline{\mathbb Q}}{{\rm Spec}\,\mathbb C}$
(see \eqref{e15}, \eqref{e15b}). The parabolic vector bundle defined by
$\varphi_*V$ (see \eqref{e15b}) equipped with the parabolic structure of a direct image will be denoted by $(\varphi_*V)_*$.
The pulled back parabolic vector bundle $\varphi^*(\varphi_*V)_*$ has the following description:
\begin{equation}\label{e18}
\varphi^*(\varphi_*V)_*\,\,=\,\, \bigoplus_{g\in \Gamma} g^*V,
\end{equation}
where $\Gamma$ is the Galois group in \eqref{e16} (see \cite[p.~19566, Proposition 4.2(2)]{AB}). In particular,
$\varphi^*(\varphi_*V)_*$ has the trivial parabolic structure, in other words, $\varphi^*(\varphi_*V)_*$ has
no nonzero parabolic weights (the underlying vector bundle is the one in the right-hand side of \eqref{e18}).

Assume that the quasiparabolic filtrations of the parabolic bundle $W_*$ in \eqref{e23} are
defined over $\overline{\mathbb Q}$; recall that $W$ is base change, to $\mathbb C$, of $\mathcal W$
using $\Psi$ in \eqref{e14}. So $W_*$ is the base change, to $\mathbb C$, of
a parabolic structure on the vector bundle $\mathcal W$. Consider the pulled back parabolic vector bundle
$\varphi^*W_*$. From the construction of $\varphi^*W_*$ (see \cite[Section 3]{AB}) it follows immediately
that $\varphi^*W_*$ has the trivial parabolic structure (it has no nonzero parabolic weights).

{}From \eqref{e17a} it follows immediately that $f\circ \gamma\,=\, \varphi$ (see \eqref{e15b}). From this
it is deduced that the parabolic vector bundle $\varphi^*W_*$ is a subbundle of the
parabolic vector bundle $\varphi^*(\varphi_*(\gamma^*E)_*)$, where $\gamma$ is the map in \eqref{e15b}
and $E$ is the vector bundle in \eqref{e3b} (see the proof of Proposition 4.3 of \cite[p.~19567]{AB}). From
\eqref{e18} we know that 
$$
\varphi^*(\varphi_*(\gamma^*E)_*)\,\,\simeq\,\, \bigoplus_{g\in \Gamma} g^*\gamma^*E,
$$
and the parabolic structure of $\varphi^*(\varphi_*(\gamma^*E)_*)$ is the trivial one.
So the parabolic structure of the parabolic subbundle $\varphi^*W_*\, \subset\, \varphi^*(\varphi_*(\gamma^*E)_*)$
is also the trivial one; this was already noted above. Consequently, we have
\begin{equation}\label{e18a}
\varphi^*W_*\,\,\, \subset\,\,\, \bigoplus_{g\in \Gamma} g^*\gamma^*E
\end{equation}
is a subbundle.

We will explicitly describe the subbundle in \eqref{e18a}.

Let $G\,:=\, \text{Gal}(\gamma_0)\,=\, \text{Aut}(Y_0/X_0)$ be the Galois group of the map
$\gamma_0$ in \eqref{e17}. So $G$ is a subgroup
of $\Gamma$ in \eqref{e16}, and $X_0\,=\, Y_0/G$. Note that $G$ is a normal subgroup of $\Gamma$ if and
only if the map $f_0$ is (ramified) Galois. We also have $G\,=\, \text{Gal}(\gamma)$ (see \eqref{e15b}).
There is a natural action of $G$ on $\gamma^*E$ over the action of
the Galois action of $G$ on $Y$. Take any $w\, \in\, (\gamma^*E)_y$, $y\,\in\, Y$, and any $h\,\in\, G$.
The point of $(\gamma^*E)_{h(y)}$ to which $w$ is taken by the action of $h$ will be denoted by $h\cdot w$.
The action of $G$ on $\gamma^*E$ (over the action of $G$ on $Y$)
produces an action of $G$ on
\begin{equation}\label{e19}
{\mathcal E}\,\,:=\,\, \bigoplus_{g\in \Gamma} g^*\gamma^*E
\end{equation}
over the trivial action of $G$ on $Y$. We will explicitly describe the action of $G$ on
the vector bundle $\mathcal E$ in \eqref{e19}. Take any point $y\,\in\, Y$. The fiber of
$\mathcal E$ over $y$ is
$$
{\mathcal E}_y\,\, =\,\, \bigoplus_{g\in \Gamma} \gamma^* E_{g(y)}.
$$
Take any element $\bigoplus_{g\in \Gamma} w_g\, \in\, \bigoplus_{g\in \Gamma} \gamma^* E_{g(y)}$, where $w_g\,
\in\, \gamma^* E_{g(y)} \,=\, E_{\gamma(g(y))}$. The action of any $h\, \in\, G$ sends $\bigoplus_{g\in \Gamma} w_g$
to $\bigoplus_{g\in \Gamma} h\cdot w_{h^{-1}g}$. The subbundle in \eqref{e18a} has the following
description:
\begin{equation}\label{e20}
\varphi^*W_*\,\,=\,\, {\mathcal E}^G\,\, \subset\,\, {\mathcal E}
\,\,= \,\, \bigoplus_{g\in \Gamma} g^*\gamma^*E,
\end{equation}
where ${\mathcal E}^G$ is the invariant subbundle (meaning the subbundle fixed pointwise) for the above
action of $G$ on ${\mathcal E}$.

Using \eqref{e20} we will show that $\gamma^*E$ is a direct summand of $\varphi^*W_*$.

Fix a subset ${\mathbb S}\, \subset\, \Gamma$ such that
\begin{itemize}
\item the following composition of maps is a bijection:
\begin{equation}\label{e21}
{\mathbb S}\,\,\hookrightarrow\,\, \Gamma\,\, \longrightarrow\,\, \Gamma/G,
\end{equation}
where $\Gamma\, \longrightarrow\,\Gamma/G$ is the quotient map to the right quotient space $\Gamma/G$ (as
mentioned before, in general $G$ is not a normal subgroup of $\Gamma$), and

\item ${\mathbb S}\cap G \,=\, \{e\}$ (the identity element of $\Gamma$).
\end{itemize}
{} From \eqref{e20} it follows that the subbundle $\varphi^*W_* \,\subset \,
\bigoplus_{g\in \Gamma} g^*\gamma^*E$ is isomorphic to the direct sum $\bigoplus_{g\in {\mathbb S}}
g^*\gamma^*E$, where $\mathbb S$ is the subset \eqref{e21}. In fact, we have an isomorphism
\begin{equation}\label{e22}
\Phi\,\, :\,\, \varphi^*W_*\,\, \longrightarrow\,\, \bigoplus_{g\in {\mathbb S}} g^*\gamma^*E
\end{equation}
which is composition of the inclusion map $\varphi^*W_*\, \hookrightarrow\, \bigoplus_{g\in \Gamma} g^*\gamma^*E$
(see \eqref{e20}) with the natural projection
$$
\bigoplus_{g\in \Gamma} g^*\gamma^*E\,\, \longrightarrow\,\, \bigoplus_{g\in {\mathbb S}} g^*\gamma^*E
$$
defined by the inclusion map ${\mathbb S}\, \hookrightarrow\, \Gamma$. Let
$\varepsilon \,\in\,{\mathbb S}$ is the unique element that projects to $eG\,\in\, \Gamma/G$, where
$e\, \in\, G$ is the identity element, under the composition of maps in \eqref{e21}; so $\varepsilon\,=\,
G\bigcap {\mathbb S}$. Note that
$\varepsilon^* \gamma^*E$ is canonically identified with $\gamma^*E$ because $\varepsilon\,\in\, G\,=\, \text{Gal}(\gamma)$.
Since the vector bundle $\gamma^*E\,=\, \varepsilon^* \gamma^*E$ is isomorphic to a direct summand of $\bigoplus_{g\in
{\mathbb S}} g^*\gamma^*E$, using the isomorphism $\Phi$ in \eqref{e22} we conclude that $\gamma^*E$ is isomorphic to a direct
summand of the holomorphic vector bundle $\varphi^*W_*$.

Recall that $\varphi^*W_*$ is the base change, to $\mathbb C$, of a vector bundle defined over $Y_0/\overline{\mathbb Q}$.
Since $\gamma^*E$ is isomorphic to a direct summand of the holomorphic vector bundle $\varphi^*W_*$, from Lemma
\ref{lem-f} (this lemma is proved below) it follows that $\gamma^*E$ is isomorphic to the base change, to $\mathbb C$, of a vector bundle 
${\mathcal V}$ on $Y_0$. Fix an isomorphism
$$
\Psi\,\, :\,\, \gamma^*E\,\, \longrightarrow\,\,{\mathcal V}\otimes_{\overline{\mathbb Q}}{\mathbb C}.
$$
Consider the corresponding isomorphism
\begin{equation}\label{z3}
\gamma_*\Psi\,\, :\,\, \gamma_*\gamma^*E\,\, \longrightarrow\,\,\gamma_*({\mathcal V}\otimes_{\overline{\mathbb Q}}{\mathbb C})
\,\,=\,\, ((\gamma_0)_*{\mathcal V})\otimes_{\overline{\mathbb Q}}{\mathbb C},
\end{equation}
where $\gamma_0$ is the map in \eqref{e17}. By the projection formula,
\begin{equation}\label{z4}
\gamma_*\gamma^*E\,\,=\,\, E\otimes \gamma_*{\mathcal O}_Y.
\end{equation}
Consider the subbundle ${\mathcal O}_X\, \subset\, \gamma_*{\mathcal O}_Y$. It is a direct summand of $\gamma_*
{\mathcal O}_Y$, meaning there is a subbundle ${\mathcal K}\,\subset\, \gamma_*{\mathcal O}_Y$ such that $\gamma_*
{\mathcal O}_Y \,=\, {\mathcal O}_X\oplus {\mathcal K}$. Consequently, from \eqref{z4} it follows that $E$ is a direct
summand of $\gamma_*\gamma^*E$. So using the isomorphism $\gamma_*\Psi$ in \eqref{z3} we conclude that
$E$ is a direct summand of $((\gamma_0)_*{\mathcal V})\otimes_{\overline{\mathbb Q}}{\mathbb C}$. Now using Lemma \ref{lem-f}
(this lemma is proved below)
it follows immediately that $E$ is isomorphic to the base change, to $\mathbb C$, of a vector bundle on $X_0/\overline{\mathbb Q}$
(see \eqref{e1a}).
\end{proof}

\subsection{Indecomposability and base change}

Let $M_0$ be an irreducible smooth projective curve defined over $\overline{\mathbb Q}$ and $E_0$ a
vector bundle on $M_0$. Consider the corresponding algebraic vector bundle $E\,=\, E_0\otimes_{\overline{\mathbb Q}}
{\mathbb C}$ over the complex projective curve 
$$
M\,\,:=\,\, M_0\times_{{\rm Spec}\,\overline{\mathbb Q}} {\rm Spec}\,\mathbb{C}.
$$

\begin{lemma}\label{lem-f}
Let $V\, \subset\, E$ be a complex algebraic subbundle of positive rank satisfying the condition that
there is another algebraic subbundle $F\, \subset\, E$ such that the natural homomorphism
\begin{equation}\label{z1}
V\oplus F\,\,\longrightarrow\,\, E
\end{equation}
is an isomorphism. Then there is a vector bundle $V_0$ on $M_0$ such that $V$ is isomorphic to the base change
$V_0\otimes_{\overline{\mathbb Q}}{\mathbb C}$ of $V_0$ to $\mathbb C$.
\end{lemma}

\begin{proof}
Let $W_i\, \longrightarrow\, M_0$ be indecomposable vector bundles such that
\begin{equation}\label{z2}
E_0\,\,=\,\, \bigoplus_{i=1}^r W_i.
\end{equation}
For any $1\, \leq\, i\, \leq\, r$, let
$$
{\mathcal W}_i\,\,:=\,\, W_i\otimes_{\overline{\mathbb Q}}{\mathbb C} \,\, \longrightarrow\,\, M
$$
be the base change, to $\mathbb C$, of $W_i$. We will show that each ${\mathcal W}_i$ is also indecomposable. For this consider the
homomorphism
$$
\Phi_0\,:\, H^0(M_0,\, \text{End}(W_i))\, \longrightarrow\, H^0(M_0,\, {\mathcal O}_{M_0})\,=\, \overline{\mathbb Q},\ \ \, A
\, \longmapsto\, {\rm trace}(A).
$$
Since $W_i$ is indecomposable, $\text{kernel}(\Phi_0)$ is a nilpotent algebra \cite[p.~201, Proposition 16]{At2} (while
this proposition is stated to $\mathbb C$, its proof is valid for $\overline{\mathbb Q}$). Since
$$
H^0(M,\, \text{End}({\mathcal W}_i))\,\,=\,\, H^0(M_0,\, \text{End}(W_i))\otimes_{\overline{\mathbb Q}}{\mathbb C},
$$
it follows that the kernel of the homomorphism
$$
\Phi\,:\, H^0(M,\, \text{End}({\mathcal W}_i))\, \longrightarrow\, H^0(M,\, {\mathcal O}_{M})\,=\, {\mathbb C},\ \ \, A
\, \longmapsto\, {\rm trace}(A)
$$
coincides with $\text{kernel}(\Phi_0)\otimes_{\overline{\mathbb Q}}{\mathbb C}$. As $\text{kernel}(\Phi_0)$ is a nilpotent
algebra, we conclude that
$$
\text{kernel}(\Phi)\,\,=\,\, \text{kernel}(\Phi_0)\otimes_{\overline{\mathbb Q}}{\mathbb C}
$$
is also a nilpotent algebra. This implies that ${\mathcal W}_i$ is indecomposable (see 
\cite[p.~201, Proposition 16]{At2}).

Therefore, the decomposition 
\begin{equation}\label{b1}
E\,\,=\,\, \bigoplus_{i=1}^r {\mathcal W}_i.
\end{equation}
given by base change, to $\mathbb C$, of the decomposition in \eqref{z2} is a decomposition of $E$ into a
direct sum of indecomposable vector bundles.

It may be mentioned that choosing a decomposition of $E_0$ (respectively, $E$) into a direct sum of
indecomposable vector bundles is equivalent to choosing a maximal torus in the group $\text{Aut}(E_0)$
(respectively, $\text{Aut}(E)$) defined by all automorphisms of $E_0$ (respectively, $E$). Note that
$\text{Aut}(E_0)$ (respectively, $\text{Aut}(E)$) is a nonempty Zariski open subset of the affine space
$H^0(M_0,\, \text{End}(E_0))$ (respectively, $H^0(M,\, \text{End}(E))$). Since $\text{Aut}(E)$ is the
base change of $\text{Aut}(E_0)$ to $\mathbb C$, the base change of a maximal torus of
$\text{Aut}(E_0)$ to $\mathbb C$ is a maximal torus of $\text{Aut}(E)$. Therefore, \eqref{b1}
is a decomposition of $E$ into a direct sum of indecomposable vector bundles.

The given condition that the homomorphism
in \eqref{z1} is an isomorphism implies that $V$ is isomorphic to the direct sum of some ${\mathcal W}_i$, in other words,
after reordering the indices $\{1,\, \cdots,\, r\}$,
$$
V\,\,=\,\, \bigoplus_{i=1}^s {\mathcal W}_i
$$
for some $1\, \leq\, s\, \leq\, r$ \cite[p.~315, Theorem 3]{At1}. So we have
$$
V\,\,=\,\, \bigoplus_{i=1}^s (W_i\otimes_{\overline{\mathbb Q}}{\mathbb C})\,\,=\,\,
\left(\bigoplus_{i=1}^s W_i\right)\otimes_{\overline{\mathbb Q}}{\mathbb C}.
$$
This completes the proof.
\end{proof}

\begin{remark}\label{rem1}
The key point in Lemma \ref{lem-f} is that $\overline{\mathbb Q}$ is an algebraically closed subfield
of $\mathbb C$. For example, the lemma is not valid if $\overline{\mathbb Q}$ is replaced by $\mathbb R$.
To give an example, take $M_0$ to be the anisotropic conic in ${\mathbb P}^2_{\mathbb R}$ defined by the
equation $X^2+Y^2+Z^2\,=\, 0$. Let $E_0$ be the unique nontrivial extension of $TM_0$ by
${\mathcal O}_{M_0}$. Then $E\,=\, E_0\otimes_{\mathbb R} {\mathbb C}$ on 
${\mathbb C}{\mathbb P}^1\,= \, M_0\times_{{\rm Spec}\, {\mathbb R}} {\rm Spec}\,\mathbb{C}$ decomposes
as ${\mathcal O}_{{\mathbb C}{\mathbb P}^1}(1)\oplus {\mathcal O}_{{\mathbb C}{\mathbb P}^1}(1)$.
But ${\mathcal O}_{{\mathbb C}{\mathbb P}^1}(1)$ is not isomorphic to the base change of any line bundle 
on $M_0$.
\end{remark}

\section*{Acknowledgements}

The second author would like to thank Max Plank Institute, Bonn for a visiting position in May-June 2023, where
he learnt about this question from Motohico Mulase. The authors thank him for posing the question.
The first author is partially supported by a J. C. Bose Fellowship (JBR/2023/000003).


\end{document}